\documentclass[11pt]{amsart} %

\usepackage{float}

\usepackage{epsfig}

\usepackage{array,amsmath, enumerate,  url, psfrag}
\usepackage{amssymb, amsaddr, fullpage}
\usepackage{graphicx,subfigure}
\usepackage[usenames,dvipsnames]{pstricks}
\usepackage{color}

\restylefloat{figure}
 \makeatother

\newtheorem{theorem}{Theorem}[section]
\newtheorem{lemma}[theorem]{Lemma}

\newtheorem{satz}{thm1}

\theoremstyle{definition}
\newtheorem{defn}[satz]{Definition}

\title{Analogue of DP-coloring on variable degeneracy and its applications on list vertex-arboricity and DP-coloring}
\author{Pongpat Sittitrai$^{1}$ \hskip 0.2in Kittikorn Nakprasit$^{1}$}

\address{
$^{1}$\small Department of Mathematics, Faculty of Science, Khon Kaen University, 40002, Thailand.}

\begin{document}

\maketitle

\begin{center}{\bf Abstract}\end{center}
\indent\indent 

In \cite{listnoC3adjC4}), Borodin and Ivanova proved that every planar graph 
without $4$-cycles adjacent to $3$-cycles is list vertex $2$-aborable. 
In fact, they proved a more general result in terms of variable degeneracy. 
Inspired by these results and DP-coloring which becomes a widely studied topic, 
we introduce a generalization on variable degeneracy including list vertex arboricity. 
We use this notion to extend a general result by Borodin and Ivanova. 
Not only that this theorem implies results 
about planar graphs without $4$-cycles adjacent to $3$-cycle by Borodin and Ivanova, 
it also implies many other results including 
a result by Kim and Yu  \cite{KimY} 
that every planar graph without $4$-cycles adjacent to $3$-cycles is DP-$4$-colorable. 


\section{Introduction}
Every graph in this paper is finite, simple, and undirected. 
We let $V(G)$ denote the vertex set and $E(G)$ denote edge set of a graph $G.$ 
For $U \subseteq V(G),$ we let $G[U]$ denote 
the subgraph of $G$ induced by $U.$ For $X, Y \subseteq V(G)$ 
where $X$ and $Y$ are disjoint, we let $E_G(X,Y)$ be the set 
of all edges in $G$ with one endpoint in $X$ and the other in $Y.$ 

The \emph{vertex-arboricity} $va(G)$ of a graph $G$ is the minimum number of subsets in which  
$V(G)$ can be partitioned so that each subset induces a forest. 
This concept was introduced by Chartrand, Kronk, and Wall \cite{1introgeq3} as \emph{point-arboricity}.   
They also proved that $va(G)\leq3$ for every planar graph $G$. 
Later, Chartrand and Kronk \cite{giveva3}  proved that this bound is sharp by providing  
an example of a planar graph $G$ with $va(G)= 3.$   
It was shown that determining the vertex-arboricity of a graph is NP-hard by Garey and Johnson \cite{nphard}
and determining whether $va(G)\leq2$ is NP-complete for maximal planar graphs $G$ 
by Hakimi and Schmeichel \cite{npcomplete}. 
Some researches in this topic are as follows. 
 
Raspaud and Wang \cite{PwoC3456Va2} showed that 
$va(G)\leq\lceil\frac{k+1}{2}\rceil$ for every $k$-degenerate graph $G$. 
It was proved that every planar graph $G$ has $va(G)\leq2$ 
when $G$ is without $k$-cycles for $k\in\{3,4,5,6\}$ (Raspaud and Wang \cite{PwoC3456Va2}), 
without $7$-cycles (Huang, Shiu, and Wang \cite{PwoC7}), 
without intersecting $3$-cycles (Chen, Raspaud, and Wang  \cite{PwointersectC3}),   
without chordal $6$-cycles (Huang and  Wang \cite{PwoC6chord}), or 
without intersecting $5$-cycle (Cai, Wu, and Sun \cite{PwointersectC5}). 
 
The concept of list coloring was independently introduced 
by Vizing \cite{Vizing} and by Erd\H os, Rubin, and Taylor \cite{Erdos}. 
A \emph{$k$-assignment} $L$ of a graph $G$ assigns a list $L(v)$ (a set of colors) 
with $|L(v)|=k$ to each vertex $v$ of  $G$. 
A graph $G$ is \emph{$L$-colorable} if there is a proper coloring $c$ where $c(v)\in L(v)$. 
If $G$ is $L$-colorable for each $k$-assignment $L$, then we say $G$ is \emph{$k$-choosable}. 
The \emph{list chromatic number} of $G,$ denoted by $\chi_{l}(G),$ 
is the minimum number $k$ such that $G$ is $k$-choosable. 
 
Borodin, Kostochka, and Toft \cite{definelistabori}  introduced list vertex arboricity 
which is list version of vertex arboricity.   
We say that $G$ has an \emph{$L$-forested-coloring} $f$ for a set $L = \{L(v)|v \in V(G)\}$ 
if  one can choose $f(v) \in L(v)$ for each vertex $v$ 
so that the subgraph induced by vertices with the same color is a forest. 
We say that $G$ is \emph{list vertex $k$-arborable} if $G$ has an $L$-forested-coloring for each 
$k$-assignment $L.$  
The \emph{list vertex arboricity} $a_l(G)$ is defined to be  the minimum $k$ such that 
$G$ is list vertex $k$-arborable. 
Obviously, $a_l(G) \geq va(G)$ for every graph $G$. 

It was proved that every planar graph $G$ is list vertex $2$-aborable  
when $G$ is without $k$-cycles for $k\in\{3,4,5,6\}$   (Xue and Wu \cite{PwoC3456lVa2}),
with no $3$-cycles at distance less than $2$ (Borodin and Ivanova \cite{listnointersectC3}),  
or without $4$-cycles adjacent to $3$-cycles (Borodin and Ivanova \cite{listnoC3adjC4}).

Dvo\v{r}\'{a}k and Postle \cite{DP} introduced a generalization 
of list coloring in which they called a \emph{correspondence coloring}. 
But following Bernshteyn, Kostochka, and Pron \cite{BKP},  
we call it a \emph{DP-coloring}.   


\begin{defn}\label{cover}
Let $L$ be an assignment of a graph $G.$ 
We call $H$ a \emph{cover} of $G$ 
if it satisfies all the followings:\\ 
(i) The vertex set of $H$ is $\bigcup_{u \in V(G)}(\{u\}\times L(u)) =
\{(u,c): u \in V(G), c \in L(u) \};$\\
(ii) $H[\{u\}\times L(u)]$ is a complete graph for each $u \in V(G);$\\
(iii) For each $uv \in E(G),$ 
the set $E_H(\{u\}\times L(u), \{v\}\times L(v))$ is a matching (maybe empty);\\
(iv) If $uv \notin E(G),$ then no edges of $H$ connect  
$\{u\}\times L(u)$ and  $\{v\}\times L(v).$   
\end{defn}

\begin{defn}\label{DP} 
An $(H,L)$-coloring of $G$ is an independent set in 
a cover $H$ of $G$ with size $|V(G)|.$ 
We say that a graph is \emph{DP-$k$-colorable} if $G$ has 
an $(H,L)$-coloring for each $k$-assignment $L$ and each cover $H$ of $G.$   
The \emph{DP-chromatic number} of $G,$ denoted by $\chi_{DP}(G),$ 
is the minimum number $k$ such that $G$ is DP-$k$-colorable. 
\end{defn}

If we define edges on $H$ to match exactly the same colors in $L(u)$ and $L(v)$ 
for each $uv \in E(G),$
then $G$ has an $(H,L)$-coloring if and only 
if $G$ is $L$-colorable.  
Thus DP-coloring is a generalization of list coloring.  
Moreover, $\chi_{DP}(G) \geq \chi_l(G).$ 
In fact, the difference of $\chi_{DP}(G)$ of $\chi_l(G)$  can be arbitrarily large. 
For graphs with average degree $d,$ Bernshteyn \cite{Bern} 
showed that   $\chi_{DP}(G) = \Omega(d /\log d),$   
whereas Alon \cite{Alon} showed that   $\chi_l(G) = \Omega(\log d).$ 

Dvo\v{r}\'{a}k  and Postle \cite{DP} observed that 
$\chi_{DP}(G) \leq 5$ for every planar graph $G.$  
This extends a seminal result by Thomassen \cite{Tho} on list colorings. 
On the other hand, Voigt \cite{Vo1}  gave an example of a planar graph 
which is not $4$-choosable (thus not DP-$4$-colorable). 
Kim and Ozeki \cite{KimO} showed that planar graphs without $k$-cycles 
are DP-$4$-colorable for each $k =3,4,5,6.$  
Kim and Yu  \cite{KimY} extended the result on $3$- and $4$-cycles 
by showing that planar graphs without $3$-cycles adjacent to $4$-cycles  
are DP-$4$-colorable. 

Inspired by $DP$-coloring and list-forested-coloring, 
we define a generalization of list-forested-coloring as follows.

\begin{defn}\label{DPrelax} 
Let $H$ be a cover of a graph $G$ with a list assignment $L.$ 
A \emph{representative set} $S$ of $G$ is a set of vertices 
in $H$ such that \\
(1) $|S|=|V(G)|$ and  \\
(2) $u \neq v$ for any two different members $(u, c)$ and $(v,c')$ in $S.$ 

A \emph{representative graph} $G_S$ is defined to be the graph 
obtained from $G$ and a representative set $S$ such that 
 vertices $u$ and $v$ are adjacent in $G_S$ if and only if 
$(u,i)$ and $(v,j)$ are in $S$ and both are adjacent in $H.$ 

A \emph{DP-forested-coloring} of $(G,H)$ is a representative set $S$ 
such that the representative graph $G_S$ is a  forest. 
We say that a graph is \emph{DP-vertex-$k$-aborable} if $G$ has 
a DP-forested-coloring of $(G,H)$ for each $k$-assignment $L$ and each cover $H$ of $G.$  
\end{defn}

If we define edges on $H$ to match exactly the same colors in $L(u)$ and $L(v)$ 
for each $uv \in E(G),$ then $G$ has a DP-forested-coloring for $G$ and $H$ 
if and only if  
$G$ has an $L$-forested-coloring. Note that $G$ has an $(H,L)$-coloring 
if and only if $G$ has a representative set $S$ such that $G_S$ has no edges. 

In \cite{listnoC3adjC4}), Borodin and Ivanova proved that every planar graph 
without $4$-cycles adjacent to $3$-cycle is list vertex $2$-aborable. 
In fact, they proved a more general result in which we explain later. 
Inspired by these results, we prove that every planar graph 
without $4$-cycles adjacent to $3$-cycles is DP-vertex-$2$-aborable. 
We also prove a theorem that extends a general result by Borodin and Ivanova. 
Among many consequences, this theorem implies a result by Kim and Yu  \cite{KimY} 
that every planar graph without $4$-cycles adjacent to $3$-cycle is DP-$4$-colorable. 

We note that results in \cite{listnoC3adjC4} are proved by means of 
a partition of the vertex set into desired sets. 
But representative sets and representative graphs cannot be considered as partitions. 
Thus we need different techniques to prove our results. 

\section{Main Results} 

Some definitions are required to understand the main results and the proofs. 
A graph $G$ is \emph{strictly} $k$-\emph{degenerate} for a positive integer $k$ 
if every subgraph $G'$ has a vertex $v$ with $d_{G'}(v) < k.$ 
Thus a strictly $1$-degenerate graph is an edgeless graph 
and a strictly $2$-degenerate graph is a forest. 
Note that vertices in a strictly $k$-degenerate can be removed in an order 
that each vertex at the time of removing is adjacent to less than $k$ remaining vertices. 
Now let $f$ be a function from $V(G)$ to the set of positive integers. 
A graph $G$ is \emph{strictly} $f$-\emph{degenerate} 
if every subgraph $G'$ has a vertex $v$ with $d_{G'}(v) < f(v).$ 

Now let $f_i, i \in \{1,\ldots,s\},$ be a function 
from $V(G)$ to the set of nonnegative integers. 
An $(f_1,\ldots,f_s)$-\emph{partition} of a graph $G$ 
is a partition of $V(G)$ into $V_1,\ldots, V_s$ such that 
an induced subgraph $G[V_i]$ is strictly $f_i$-degenerate 
for each $i\in \{1,\ldots,s\}.$ 
A $(k_1,\ldots,k_s)$-\emph{partition} where $k_i$ is a constant for each $i\in \{1,\ldots,s\}$ 
is an  $(f_1,\ldots,f_s)$-partition such that $f_i(v)=k_i$ for each vertex $v.$ 
We say that $G$ is $(f_1,\ldots,f_s)$-\emph{partitionable} if 
$G$ has  an $(f_1,\ldots,f_s)$-partition. 
Let $c$ be a function from $V(G)$ to the set of positive integers. 
Define $f_c$ from $f_i, i \in \{1,\ldots,s\},$ and $c$ by $f_c(v)=f_{c(v)}(v).$ 
Define $G_c$ to be a graph obtained from $G$ and $c$ 
such that $V(G_c)=V(G)$ while vertices $u$ and $v$ are adjacent in $G_c$ 
if and only if $u$ and $v$ are adjacent in $G$ and $c(u)=c(v).$ 
Thus a graph $G$ is $(f_1,\ldots,f_s)$-partitionable 
if and only if there is a function $c$ such that $G_c$ is strictly $f_c$-degenerate. 
By Four Color Theorem \cite{1111partition}, every planar graph is $(1,1,1,1)$-partitionable. 
However, Chartrand and Kronk \cite{giveva3} constructed planar graphs which are not $(2,2)$-partitionable. 
Even stronger, Wegner \cite{not211} showed that 
there exists a planar graph which is not $(2,1,1)$-partitionable. 
Thus it is of interest to find sufficient conditions for planar graphs 
to be $(1,1,1,1)$-, $(2,1,1)$-, or $(2,2)$-partitionable.  

Borodin, Kostochka, and Toft \cite{definelistabori} observed that the notion of 
 $(f_1,\ldots,f_s)$-partition can be applied to problems in list coloring 
and list vertex arboricity. 
Since $v$ cannot be strictly $0$-degenerate, the condition 
that $f_i(v)=0$ is equivalent to $v$ cannot be colored by  $i.$ 
In other words, $i$ is not in the list of $v.$    
Thus the case of $f_i \in \{0,1\}$ corresponds to list coloring, 
and one of  $f_i \in \{0,1\}$ corresponds to $L$-forested-coloring.  
On the other hand, Voight \cite{Vo1} showed that 
there exists a planar graph that is not $4$-choosable. 
Naturally, it is also interesting to find  sufficient conditions for planar graphs 
to be $4$-choosable or list vertex $2$-aborable. 
Borodin and Ivanova  \cite{listnoC3adjC4} obtained a general result 
which implies planar graphs are $4$-choosable and list vertex $2$-aborable.  

\begin{theorem}\label{main0} Every planar graph without $4$-cycles adjacent to $3$-cycles is 
$(f_1, \ldots, f_s)$-partitionable 
if $s \geq 2,$ $f_1(v)+\cdots+f_s(v)\geq 4$ for each vertex $v,$ 
and $f_i(v)\in \{0,1,2\}$ for each $v$ and $i.$ 
\end{theorem}

We extend the concept of DP-coloring to $(f_1,\ldots,f_s)$-partition as follows. 
Let $H$ be a cover of $G$ with the list $\{1,\ldots,s\}$ for every vertex and 
$R$ be a representative set.  
Define $f_R(v)$ to equal $f_i(v)$ where $(v,i) \in R.$   
We say that a graph $G$ is \emph{DP-$(f_1,\ldots,f_s)$-colorable}  
if we can find a representative set $R$ for every cover $H$ of $G$ such that 
$G_R$ is strictly $f_R$-degenerate. 
We say that $R$ is a \emph{DP-$(f_1,\ldots,f_s)$-coloring} 
If we define edges on $H$ to match exactly the same colors 
for each $uv \in E(G),$ then a $(f_1,\ldots,f_s)$-partition exists  
if and only if a DP-$(f_1,\ldots,f_s)$-coloring exists.  
Thus $(f_1,\ldots,f_s)$-partition is a special case of DP-$(f_1,\ldots,f_s)$-coloring.

\begin{lemma}\label{C35} \cite{4choosable} Every planar graph without $4$-cycles adjacent to $3$-cycles 
contains a configuration, say $F$,  
which is a $6$-cycle $x_1\ldots x_6$ with a chord $x_1x_5$ such that $d(x_i) =4$ 
for each $i \in \{1, \ldots, 6\}.$ 
\end{lemma}

Note that a DP-$(2,2)$-coloring is equivalent to a DP-forested-coloring. 

\begin{theorem}\label{main1} Every planar graph without $4$-cycles adjacent to $3$-cycles 
is DP-vertex-$2$-aborable. 
\end{theorem} 
\begin{proof} 
Suppose that $G$ with a cover $H$ is a minimal counterexample. 
Since $G$ does not have $4$-cycles adjacent to $3$-cycles, 
$G$ has a configuration $F$ as in Lemma \ref{C35}. 
Since a $4$-cycle is not adjacent to a $3$-cycle in $G,$ we obtain that  $F$ is an induced subgraph of $G.$  
By minimality, there is a DP-$(2,2)$-coloring $R'$ on $G - \{x_1,\ldots,x_6\}.$ 
It remains to show that we can extend a DP-$(2,2)$-coloring to $G.$ 

For each $x_k \in V(F)$ and $i \in \{1,2\},$ we put $f^*_i(x_k)$ equals $2$ minus 
the number of $(v,j) \in R'$ such that $(v,j)$ and $(x_k,i)$ are adjacent in $H.$ 

Note that if $F$ has a DP-$(f^*_1,f^*_2)$-coloring $R^*$, 
then one can obtain a desired DP-$(2,2)$-coloring on 
$G$ which can be seen from the removal such that  
 we remove vertices in $\{x_1,\ldots, x_6\}$ 
(in an order according to  $R^*$), 
and then we remove the vertices in $G- \{x_1,\ldots, x_6\}$ 
(in an order according $R'$).   

By $(f_1(x_j),f_2(x_j)) =(2,2)$ and the definition of $f^*_i(x_j),$  
we have  $\{f^*_1(x_1),f^*_2(x_1)\} =\{f^*_1(x_5),f^*_2(x_5)\} =\{1,2\}.$  
Also, we have $f^*_1(x_j)+f^*_2(x_j)\} \geq 2.$  
We will consider an inequality as an equality because 
the remaining situations are easier. 

CASE 1: $f_i^*(x_k) \geq 1$ for each $i\in \{1,2\}$ and $k \in \{1, \ldots, 6\}.$\\ 
From above, we have $(f	_1^*(x_1),f_2^*(x_1))= (1,2)$ or $(2,1)$ 
and $(f	_1^*(x_i),f_2^*(x_i))= (1,1)$ for each $i \in \{2,3,4,6\}.$ 
By symmetry, we assume $(f	_1^*(x_5),f_2^*(x_5))= (1,2).$ 
Since the names of colors can be interchanged, 
we assume further that $(x_k,i)$ and $(x_{k+1},i)$ are adjacent  in $H^*$ for each $k \in \{1,\ldots,4\}$ 
and $i \in \{1,2\}.$ 
However, the matchings from $\{(x_1,1),(x_1,2)\}$ to $\{(x_5,1),(x_5,2)\}$ 
and to $\{(x_6,1),(x_6,2)\}$ are arbitrary.  
Thus there are four non-isomorphic structure of $H^*.$ 
To illustrate  desired colorings for all four structures, 
we use Figure 1 to clarify the representation for a vertex $x_k.$ 
The single cycle means $(x_k,1)$ and the double cycle means $(x_k,2).$ 
The shade at $(x_k,1)$ indicates that we choose $(x_k,1)$ to be in a coloring $R^*.$  
Figures 2-5 show all four structures of $H^*$ with desired colorings.  

CASE 2: there exists $k$ in which $f_i^*(x_k)=0$ but $f_j^*(x_{k+1}) \geq 1$ 
where $(x_k,i)$ and $(x_{k+1},j)$ are adjacent. \\
Note that all subscripts in this case are taken in modulo $6.$  
We will apply greedy coloring in which we described later  to
 $x_{k+1},x_{k+2},\ldots,x_6,x_1,x_2,\ldots,x_{k}$ respectively.  
If we choose $(x_p,i)$ to be in $R^*$ in the process of a coloring,   
we update $f_1^*(x_{q})$ and $f_2^*(x_{q})$ of an uncolored vertex $x_q$  
by $f_j^*(x_q)  = \max \{0, f_j^*(x_q)-1 \}$ 
if $(x_p,i)$ and $(x_q,j)$ are adjacent in $H^*.$ 

First, we choose $(x_{k+1},j)$ to be in $R^*.$ 
By the condition of the case, $(f_1^*(x_k),f_2^*(x_k))$ remains the same after an update. 
Next apply greedy coloring to $x_{k+2},\ldots,x_6,x_1,x_2,\ldots,x_{k-1}$  
by choosing $(x_m,i)$ such that $f_i^*(x_m) >0$ to be in $R^*.$  
Since $f^*_1(x_j)+f^*_2(x_j) \geq d_F(x_j)$ before the process, 
one can see that a greedy coloring can be attained.  
Now at $x_k,$ we have that $(f_1^*(x_k),f_2^*(x_k)) \neq (0,0)$ 
by the choosing of $(x_{k+1},j)$ in the beginning.  
Thus we can choose $(x_k,1)$ or $(x_k,2)$ to be in $R^*$ to complete the coloring.  

Now it remains to show that every $(f^*_1,f^*_2)$ of $F$ in the beginning 
is similar to one in CASE 1 or CASE 2. 
From the observation in a paragraph before both CASE 1,  
we have $\{f^*_1(x_1),f^*_2(x_1)\} =\{f^*_1(x_5),f^*_2(x_5)\} =\{1,2\}.$  
Also, we have $f^*_1(x_j)+f^*_2(x_j)\} = 2.$  
Suppose  $(f^*_1,f^*_2)$ is not  as in CASE 2. 
Considering $(f^*_1(x_1),f^*_2(x_1)),$ we have $f^*_1(x_6)=f^*_2(x_6) = 1.$ 
Similarly, considering $(f^*_1(x_5),f^*_2(x_5)),$  we have $f^*_1(x_4)=f^*_2(x_4) = 1.$ 
Recursively, we obtain that  $f^*_1(x_i)=f^*_2(x_i) = 1$ for $i=3$ and $i=2.$ 
Thus we have the situation as in CASE 1.  
 \end{proof}

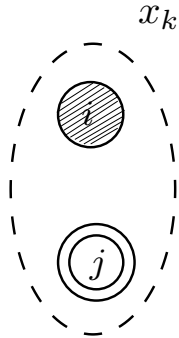
\begin{figure}[ht]\label{fig1}
\centering
\scalebox{1.5} 
{
	\begin{pspicture}(0,-1.4820312)(1.7928125,1.5220313)
	\usefont{T1}{ptm}{m}{n}
	\rput(0.79,-0.84703124){\footnotesize $j$}
	\pscircle[linewidth=0.022,dimen=outer](0.77,-0.83203125){0.35}
	\pscircle[linewidth=0.022,dimen=outer](0.77,-0.83203125){0.25}
	\psellipse[linewidth=0.022,linestyle=dashed,dash=0.16cm 0.16cm,dimen=outer](0.74,-0.18203124)(0.74,1.3)
	\pscircle[linewidth=0.022,dimen=outer](0.72,0.47796875){0.3}
	\usefont{T1}{ptm}{m}{n}
	\rput(0.69359374,0.49296874){\footnotesize $i$}
	\psline[linewidth=0.004cm](0.46,0.5779688)(0.74,0.7579687)
	\psline[linewidth=0.004cm](0.42,0.49796876)(0.82,0.7579687)
	\psline[linewidth=0.004cm](0.44,0.45796874)(0.84,0.71796876)
	\psline[linewidth=0.004cm](0.44,0.41796875)(0.9,0.6979687)
	\psline[linewidth=0.004cm](0.46,0.37796876)(0.94,0.67796874)
	\psline[linewidth=0.004cm](0.48,0.33796874)(0.96,0.6379688)
	\psline[linewidth=0.004cm](0.5,0.29796875)(0.96,0.59796876)
	\psline[linewidth=0.004cm](0.54,0.27796876)(1.0,0.55796874)
	\psline[linewidth=0.004cm](0.58,0.25796875)(1.02,0.5179688)
	\psline[linewidth=0.004cm](0.62,0.23796874)(1.0,0.45796874)
	\psline[linewidth=0.004cm](0.66,0.21796875)(1.0,0.41796875)
	\psline[linewidth=0.004cm](0.72,0.21796875)(0.96,0.35796875)
	\usefont{T1}{ptm}{m}{n}
	\rput(1.3275,1.3379687){\small $x_k$}
	\end{pspicture} 
}

\caption{$(x_k,1)$ with ${f_1}^*(x_k)=i, (x_k,2)$ with ${f_2}^*(x_k)=j$ 
and we choose $(x_k,1)$ in a coloring} 
\end{figure}

\begin{figure}[ht]\label{fig2}
\centering
\scalebox{1} 
{
	\begin{pspicture}(0,-2.86875)(8.835,2.86875)
	\pscircle[linewidth=0.022,dimen=outer](3.4521875,-1.1253124){0.35}
	\pscircle[linewidth=0.022,dimen=outer](3.4521875,-1.1253124){0.25}
	\psline[linewidth=0.004cm](3.1288955,-1.0234747)(3.3754795,-0.80715024)
	\psline[linewidth=0.004cm](3.1131608,-1.0898082)(3.491214,-0.78081673)
	\psline[linewidth=0.004cm](3.1221876,-1.1553125)(3.5421875,-0.7953125)
	\psline[linewidth=0.004cm](3.1221876,-1.2153125)(3.5821874,-0.8353125)
	\psline[linewidth=0.004cm](3.1421876,-1.2553124)(3.6421876,-0.8553125)
	\psline[linewidth=0.004cm](3.1621876,-1.2953125)(3.6821876,-0.8953125)
	\psline[linewidth=0.004cm](3.1821876,-1.3353125)(3.7221875,-0.9353125)
	\psline[linewidth=0.004cm](3.2221875,-1.3553125)(3.7421875,-0.9753125)
	\psline[linewidth=0.004cm](3.2621875,-1.3953125)(3.7621875,-1.0353125)
	\psline[linewidth=0.004cm](3.3021874,-1.4153125)(3.7821875,-1.0753125)
	\psline[linewidth=0.004cm](3.3621874,-1.4353125)(3.7821875,-1.1353126)
	\psline[linewidth=0.004cm](3.4421875,-1.4353125)(3.7621875,-1.2153125)
	\psline[linewidth=0.004cm](3.5021875,-1.4353125)(3.7221875,-1.2953125)
	\usefont{T1}{ptm}{m}{n}
	\rput(3.5057812,-1.1003125){\footnotesize $2$}
	\pscircle[linewidth=0.022,dimen=outer](5.2821875,1.9446875){0.3}
	\usefont{T1}{ptm}{m}{n}
	\rput(5.3057814,1.9596875){\footnotesize $1$}
	\psline[linewidth=0.004cm](5.0221877,2.0446875)(5.3021874,2.2246876)
	\psline[linewidth=0.004cm](4.9821873,1.9646875)(5.3821874,2.2246876)
	\psline[linewidth=0.004cm](5.0021877,1.9246875)(5.4021873,2.1846876)
	\psline[linewidth=0.004cm](5.0021877,1.8846875)(5.4621873,2.1646874)
	\psline[linewidth=0.004cm](5.0221877,1.8446875)(5.5021877,2.1446874)
	\psline[linewidth=0.004cm](5.0421877,1.8046875)(5.5221877,2.1046875)
	\psline[linewidth=0.004cm](5.0621877,1.7646875)(5.5221877,2.0646875)
	\psline[linewidth=0.004cm](5.1021876,1.7446876)(5.5621877,2.0246875)
	\psline[linewidth=0.004cm](5.1421876,1.7246875)(5.5821877,1.9846874)
	\psline[linewidth=0.004cm](5.1821876,1.7046875)(5.5621877,1.9246875)
	\psline[linewidth=0.004cm](5.2221875,1.6846875)(5.5621877,1.8846875)
	\psline[linewidth=0.004cm](5.2821875,1.6846875)(5.5221877,1.8246875)
	\usefont{T1}{ptm}{m}{n}
	\rput(5.3257813,-1.1203125){\footnotesize $1$}
	\pscircle[linewidth=0.022,dimen=outer](5.2921877,-1.1253124){0.35}
	\pscircle[linewidth=0.022,dimen=outer](5.2921877,-1.1253124){0.25}
	\usefont{T1}{ptm}{m}{n}
	\rput(7.565781,0.1196875){\footnotesize $1$}
	\pscircle[linewidth=0.022,dimen=outer](7.5221877,0.1246875){0.3}
	\pscircle[linewidth=0.022,dimen=outer](6.6921873,0.1146875){0.35}
	\pscircle[linewidth=0.022,dimen=outer](6.6921873,0.1146875){0.25}
	\psline[linewidth=0.004cm](6.3688955,0.21652527)(6.6154795,0.43284974)
	\psline[linewidth=0.004cm](6.353161,0.15019174)(6.731214,0.45918328)
	\psline[linewidth=0.004cm](6.3621874,0.0846875)(6.7821875,0.4446875)
	\psline[linewidth=0.004cm](6.3621874,0.0246875)(6.8221874,0.4046875)
	\psline[linewidth=0.004cm](6.3821874,-0.0153125)(6.8821874,0.3846875)
	\psline[linewidth=0.004cm](6.4021873,-0.0553125)(6.9221873,0.3446875)
	\psline[linewidth=0.004cm](6.4221873,-0.0953125)(6.9621873,0.3046875)
	\psline[linewidth=0.004cm](6.4621873,-0.1153125)(6.9821873,0.2646875)
	\psline[linewidth=0.004cm](6.5021877,-0.1553125)(7.0021877,0.2046875)
	\psline[linewidth=0.004cm](6.5421877,-0.1753125)(7.0221877,0.1646875)
	\psline[linewidth=0.004cm](6.6021876,-0.1953125)(7.0221877,0.1046875)
	\psline[linewidth=0.004cm](6.6821876,-0.1953125)(7.0021877,0.0246875)
	\psline[linewidth=0.004cm](6.7421875,-0.1953125)(6.9621873,-0.0553125)
	\usefont{T1}{ptm}{m}{n}
	\rput(6.7457814,0.1396875){\footnotesize $1$}
	\rput{89.35016}(7.086884,-6.9383473){\psellipse[linewidth=0.022,linestyle=dashed,dash=0.16cm 0.16cm,dimen=outer](7.0521874,0.1146875)(0.53,0.89)}
	\usefont{T1}{ptm}{m}{n}
	\rput(2.0057812,0.0996875){\footnotesize $1$}
	\pscircle[linewidth=0.022,dimen=outer](1.9721875,0.0946875){0.35}
	\pscircle[linewidth=0.022,dimen=outer](1.9721875,0.0946875){0.25}
	\pscircle[linewidth=0.022,dimen=outer](1.1621875,0.0846875){0.3}
	\usefont{T1}{ptm}{m}{n}
	\rput(1.1857812,0.0996875){\footnotesize $1$}
	\psline[linewidth=0.004cm](0.9021875,0.1846875)(1.1821876,0.3646875)
	\psline[linewidth=0.004cm](0.8621875,0.1046875)(1.2621875,0.3646875)
	\psline[linewidth=0.004cm](0.8821875,0.0646875)(1.2821875,0.3246875)
	\psline[linewidth=0.004cm](0.8821875,0.0246875)(1.3421875,0.3046875)
	\psline[linewidth=0.004cm](0.9021875,-0.0153125)(1.3821875,0.2846875)
	\psline[linewidth=0.004cm](0.9221875,-0.0553125)(1.4021875,0.2446875)
	\psline[linewidth=0.004cm](0.9421875,-0.0953125)(1.4021875,0.2046875)
	\psline[linewidth=0.004cm](0.9821875,-0.1153125)(1.4421875,0.1646875)
	\psline[linewidth=0.004cm](1.0221875,-0.1353125)(1.4621875,0.1246875)
	\psline[linewidth=0.004cm](1.0621876,-0.1553125)(1.4421875,0.0646875)
	\psline[linewidth=0.004cm](1.1021875,-0.1753125)(1.4421875,0.0246875)
	\psline[linewidth=0.004cm](1.1621875,-0.1753125)(1.4021875,-0.0353125)
	\rput{89.35016}(1.6685853,-1.5382434){\psellipse[linewidth=0.022,linestyle=dashed,dash=0.16cm 0.16cm,dimen=outer](1.6121875,0.0746875)(0.53,0.89)}
	\pscircle[linewidth=0.022,dimen=outer](3.4521875,1.1346875){0.35}
	\pscircle[linewidth=0.022,dimen=outer](3.4521875,1.1346875){0.25}
	\psline[linewidth=0.004cm](3.1288955,1.2365253)(3.3754795,1.4528497)
	\psline[linewidth=0.004cm](3.1131608,1.1701918)(3.491214,1.4791833)
	\psline[linewidth=0.004cm](3.1221876,1.1046875)(3.5421875,1.4646875)
	\psline[linewidth=0.004cm](3.1221876,1.0446875)(3.5821874,1.4246875)
	\psline[linewidth=0.004cm](3.1421876,1.0046875)(3.6421876,1.4046875)
	\psline[linewidth=0.004cm](3.1621876,0.9646875)(3.6821876,1.3646874)
	\psline[linewidth=0.004cm](3.1821876,0.9246875)(3.7221875,1.3246875)
	\psline[linewidth=0.004cm](3.2221875,0.9046875)(3.7421875,1.2846875)
	\psline[linewidth=0.004cm](3.2621875,0.8646875)(3.7621875,1.2246875)
	\psline[linewidth=0.004cm](3.3021874,0.8446875)(3.7821875,1.1846875)
	\psline[linewidth=0.004cm](3.3621874,0.8246875)(3.7821875,1.1246876)
	\psline[linewidth=0.004cm](3.4421875,0.8246875)(3.7621875,1.0446875)
	\psline[linewidth=0.004cm](3.5021875,0.8246875)(3.7221875,0.9646875)
	\usefont{T1}{ptm}{m}{n}
	\rput(3.4657812,1.1596875){\footnotesize $j$}
	\usefont{T1}{ptm}{m}{n}
	\rput(5.3257813,1.1396875){\footnotesize $1$}
	\pscircle[linewidth=0.022,dimen=outer](5.2921877,1.1346875){0.35}
	\pscircle[linewidth=0.022,dimen=outer](5.2921877,1.1346875){0.25}
	\usefont{T1}{ptm}{m}{n}
	\rput(3.4357812,1.9396875){\footnotesize $i$}
	\pscircle[linewidth=0.022,dimen=outer](3.4421875,1.9446875){0.3}
	\usefont{T1}{ptm}{m}{n}
	\rput(3.4857812,-1.9003125){\footnotesize $1$}
	\pscircle[linewidth=0.022,dimen=outer](3.4421875,-1.8953125){0.3}
	\pscircle[linewidth=0.022,dimen=outer](5.2821875,-1.9153125){0.3}
	\usefont{T1}{ptm}{m}{n}
	\rput(5.3057814,-1.9003125){\footnotesize $1$}
	\psline[linewidth=0.004cm](5.0221877,-1.8153125)(5.3021874,-1.6353126)
	\psline[linewidth=0.004cm](4.9821873,-1.8953125)(5.3821874,-1.6353126)
	\psline[linewidth=0.004cm](5.0021877,-1.9353125)(5.4021873,-1.6753125)
	\psline[linewidth=0.004cm](5.0021877,-1.9753125)(5.4621873,-1.6953125)
	\psline[linewidth=0.004cm](5.0221877,-2.0153124)(5.5021877,-1.7153125)
	\psline[linewidth=0.004cm](5.0421877,-2.0553124)(5.5221877,-1.7553124)
	\psline[linewidth=0.004cm](5.0621877,-2.0953126)(5.5221877,-1.7953125)
	\psline[linewidth=0.004cm](5.1021876,-2.1153126)(5.5621877,-1.8353125)
	\psline[linewidth=0.004cm](5.1421876,-2.1353126)(5.5821877,-1.8753124)
	\psline[linewidth=0.004cm](5.1821876,-2.1553125)(5.5621877,-1.9353125)
	\psline[linewidth=0.004cm](5.2221875,-2.1753125)(5.5621877,-1.9753125)
	\psline[linewidth=0.004cm](5.2821875,-2.1753125)(5.5221877,-2.0353124)
	\psellipse[linewidth=0.022,linestyle=dashed,dash=0.16cm 0.16cm,dimen=outer](3.4521875,-1.4653125)(0.53,0.89)
	\psellipse[linewidth=0.022,linestyle=dashed,dash=0.16cm 0.16cm,dimen=outer](5.2921877,-1.4653125)(0.53,0.89)
	\psellipse[linewidth=0.022,linestyle=dashed,dash=0.16cm 0.16cm,dimen=outer](3.4521875,1.5146875)(0.53,0.89)
	\psellipse[linewidth=0.022,linestyle=dashed,dash=0.16cm 0.16cm,dimen=outer](5.2921877,1.5346875)(0.53,0.89)
	\usefont{T1}{ptm}{m}{n}
	\rput(3.4596875,2.6846876){\small $x_1$}
	\usefont{T1}{ptm}{m}{n}
	\rput(5.2596874,2.6846876){\small $x_2$}
	\usefont{T1}{ptm}{m}{n}
	\rput(8.359688,0.1246875){\small $x_3$}
	\usefont{T1}{ptm}{m}{n}
	\rput(5.3196874,-2.6553125){\small $x_4$}
	\usefont{T1}{ptm}{m}{n}
	\rput(3.5596876,-2.6553125){\small $x_5$}
	\usefont{T1}{ptm}{m}{n}
	\rput(0.3996875,0.1646875){\small $x_6$}
	\psline[linewidth=0.024cm](2.1221876,0.4246875)(3.1221876,1.0446875)
	\psline[linewidth=0.024cm](3.7821875,1.0646875)(4.9621873,1.0646875)
	\psline[linewidth=0.024cm](5.6221876,1.0446875)(6.5221877,0.4046875)
	\psline[linewidth=0.024cm](6.5221877,-0.1753125)(5.6021876,-0.9753125)
	\psline[linewidth=0.024cm](4.9621873,-0.9553125)(3.7821875,-0.9553125)
	\psline[linewidth=0.024cm](3.1221876,-0.9553125)(2.1221876,-0.2153125)
	\psline[linewidth=0.024cm](1.1621875,0.3846875)(3.1621876,1.8646874)
	\psline[linewidth=0.024cm](3.7421875,1.8646874)(4.9821873,1.8646874)
	\psline[linewidth=0.024cm](5.5421877,1.8646874)(7.5421877,0.4246875)
	\psline[linewidth=0.024cm](7.5821877,-0.1553125)(5.5821877,-1.9353125)
	\psline[linewidth=0.024cm](4.9821873,-1.9553125)(3.7221875,-1.9753125)
	\psline[linewidth=0.024cm](3.1421876,-1.9353125)(1.1421875,-0.2153125)
	\psbezier[linewidth=0.08](3.7771952,-1.2053437)(4.157441,-1.2153381)(4.1621876,1.1847245)(3.7819414,1.1947188)
	\psbezier[linewidth=0.024](3.677195,-2.058559)(4.847579,-2.0753126)(4.8621874,1.9479338)(3.6918037,1.9646875)
	\end{pspicture} 
}
\caption{A desired coloring of $F$ with respect to this cover}
\end{figure}

\begin{figure}[ht]\label{fig3}
	\centering
	\scalebox{1} 
	{
		\begin{pspicture}(0,-2.86875)(8.835,2.86875)
		\pscircle[linewidth=0.022,dimen=outer](3.4721875,-1.0453125){0.35}
		\pscircle[linewidth=0.022,dimen=outer](3.4721875,-1.0453125){0.25}
		\psline[linewidth=0.004cm](3.1488955,-0.9434747)(3.3954797,-0.72715026)
		\psline[linewidth=0.004cm](3.1331608,-1.0098083)(3.511214,-0.70081675)
		\psline[linewidth=0.004cm](3.1421876,-1.0753125)(3.5621874,-0.7153125)
		\psline[linewidth=0.004cm](3.1421876,-1.1353126)(3.6021874,-0.7553125)
		\psline[linewidth=0.004cm](3.1621876,-1.1753125)(3.6621876,-0.7753125)
		\psline[linewidth=0.004cm](3.1821876,-1.2153125)(3.7021875,-0.8153125)
		\psline[linewidth=0.004cm](3.2021875,-1.2553124)(3.7421875,-0.8553125)
		\psline[linewidth=0.004cm](3.2421875,-1.2753125)(3.7621875,-0.8953125)
		\psline[linewidth=0.004cm](3.2821875,-1.3153125)(3.7821875,-0.9553125)
		\psline[linewidth=0.004cm](3.3221874,-1.3353125)(3.8021874,-0.9953125)
		\psline[linewidth=0.004cm](3.3821876,-1.3553125)(3.8021874,-1.0553125)
		\psline[linewidth=0.004cm](3.4621875,-1.3553125)(3.7821875,-1.1353126)
		\psline[linewidth=0.004cm](3.5221875,-1.3553125)(3.7421875,-1.2153125)
		\usefont{T1}{ptm}{m}{n}
		\rput(3.5057812,-1.1003125){\footnotesize $2$}
		\usefont{T1}{ptm}{m}{n}
		\rput(3.5,1.1196876){\footnotesize $j$}
		\pscircle[linewidth=0.022,dimen=outer](3.4521875,1.0946875){0.35}
		\pscircle[linewidth=0.022,dimen=outer](3.4521875,1.0946875){0.25}
		\psline[linewidth=0.004cm](3.1288955,1.1965252)(3.3754795,1.4128498)
		\psline[linewidth=0.004cm](3.1131608,1.1301917)(3.491214,1.4391832)
		\psline[linewidth=0.004cm](3.1221876,1.0646875)(3.5421875,1.4246875)
		\psline[linewidth=0.004cm](3.1221876,1.0046875)(3.5821874,1.3846875)
		\psline[linewidth=0.004cm](3.1421876,0.9646875)(3.6421876,1.3646874)
		\psline[linewidth=0.004cm](3.1621876,0.9246875)(3.6821876,1.3246875)
		\psline[linewidth=0.004cm](3.1821876,0.8846875)(3.7221875,1.2846875)
		\psline[linewidth=0.004cm](3.2221875,0.8646875)(3.7421875,1.2446876)
		\psline[linewidth=0.004cm](3.2621875,0.8246875)(3.7621875,1.1846875)
		\psline[linewidth=0.004cm](3.3021874,0.8046875)(3.7821875,1.1446875)
		\psline[linewidth=0.004cm](3.3621874,0.7846875)(3.7821875,1.0846875)
		\psline[linewidth=0.004cm](3.4421875,0.7846875)(3.7621875,1.0046875)
		\psline[linewidth=0.004cm](3.5021875,0.7846875)(3.7221875,0.9246875)
		\usefont{T1}{ptm}{m}{n}
		\rput(5.2657814,-1.9603125){\footnotesize $1$}
		\rput{89.35016}(7.086884,-6.9383473){\psellipse[linewidth=0.022,linestyle=dashed,dash=0.16cm 0.16cm,dimen=outer](7.0521874,0.1146875)(0.53,0.89)}
		\usefont{T1}{ptm}{m}{n}
		\rput(2.0057812,0.0996875){\footnotesize $1$}
		\pscircle[linewidth=0.022,dimen=outer](1.9721875,0.0946875){0.35}
		\pscircle[linewidth=0.022,dimen=outer](1.9721875,0.0946875){0.25}
		\pscircle[linewidth=0.022,dimen=outer](1.1621875,0.0846875){0.3}
		\usefont{T1}{ptm}{m}{n}
		\rput(1.1857812,0.0996875){\footnotesize $1$}
		\psline[linewidth=0.004cm](0.9021875,0.1846875)(1.1821876,0.3646875)
		\psline[linewidth=0.004cm](0.8621875,0.1046875)(1.2621875,0.3646875)
		\psline[linewidth=0.004cm](0.8821875,0.0646875)(1.2821875,0.3246875)
		\psline[linewidth=0.004cm](0.8821875,0.0246875)(1.3421875,0.3046875)
		\psline[linewidth=0.004cm](0.9021875,-0.0153125)(1.3821875,0.2846875)
		\psline[linewidth=0.004cm](0.9221875,-0.0553125)(1.4021875,0.2446875)
		\psline[linewidth=0.004cm](0.9421875,-0.0953125)(1.4021875,0.2046875)
		\psline[linewidth=0.004cm](0.9821875,-0.1153125)(1.4421875,0.1646875)
		\psline[linewidth=0.004cm](1.0221875,-0.1353125)(1.4621875,0.1246875)
		\psline[linewidth=0.004cm](1.0621876,-0.1553125)(1.4421875,0.0646875)
		\psline[linewidth=0.004cm](1.1021875,-0.1753125)(1.4421875,0.0246875)
		\psline[linewidth=0.004cm](1.1621875,-0.1753125)(1.4021875,-0.0353125)
		\rput{89.35016}(1.6685853,-1.5382434){\psellipse[linewidth=0.022,linestyle=dashed,dash=0.16cm 0.16cm,dimen=outer](1.6121875,0.0746875)(0.53,0.89)}
		\usefont{T1}{ptm}{m}{n}
		\rput(5.2857814,1.1196876){\footnotesize $1$}
		\pscircle[linewidth=0.022,dimen=outer](5.2921877,1.1346875){0.35}
		\pscircle[linewidth=0.022,dimen=outer](5.2921877,1.1346875){0.25}
		\usefont{T1}{ptm}{m}{n}
		\rput(3.4857812,-1.9003125){\footnotesize $1$}
		\pscircle[linewidth=0.022,dimen=outer](3.4421875,-1.8953125){0.3}
		\pscircle[linewidth=0.022,dimen=outer](5.2821875,-1.9553125){0.3}
		\psline[linewidth=0.004cm](5.0221877,-1.8553125)(5.3021874,-1.6753125)
		\psline[linewidth=0.004cm](4.9821873,-1.9353125)(5.3821874,-1.6753125)
		\psline[linewidth=0.004cm](5.0021877,-1.9753125)(5.4021873,-1.7153125)
		\psline[linewidth=0.004cm](5.0021877,-2.0153124)(5.4621873,-1.7353125)
		\psline[linewidth=0.004cm](5.0221877,-2.0553124)(5.5021877,-1.7553124)
		\psline[linewidth=0.004cm](5.0421877,-2.0953126)(5.5221877,-1.7953125)
		\psline[linewidth=0.004cm](5.0621877,-2.1353126)(5.5221877,-1.8353125)
		\psline[linewidth=0.004cm](5.1021876,-2.1553125)(5.5621877,-1.8753124)
		\psline[linewidth=0.004cm](5.1421876,-2.1753125)(5.5821877,-1.9153125)
		\psline[linewidth=0.004cm](5.1821876,-2.1953125)(5.5621877,-1.9753125)
		\psline[linewidth=0.004cm](5.2221875,-2.2153125)(5.5621877,-2.0153124)
		\psline[linewidth=0.004cm](5.2821875,-2.2153125)(5.5221877,-2.0753126)
		\psellipse[linewidth=0.022,linestyle=dashed,dash=0.16cm 0.16cm,dimen=outer](3.4521875,-1.4653125)(0.53,0.89)
		\psellipse[linewidth=0.022,linestyle=dashed,dash=0.16cm 0.16cm,dimen=outer](5.2921877,-1.4653125)(0.53,0.89)
		\psellipse[linewidth=0.022,linestyle=dashed,dash=0.16cm 0.16cm,dimen=outer](3.4521875,1.5146875)(0.53,0.89)
		\psellipse[linewidth=0.022,linestyle=dashed,dash=0.16cm 0.16cm,dimen=outer](5.2921877,1.5346875)(0.53,0.89)
		\usefont{T1}{ptm}{m}{n}
		\rput(3.4596875,2.6846876){\small $x_1$}
		\usefont{T1}{ptm}{m}{n}
		\rput(5.2596874,2.6846876){\small $x_2$}
		\usefont{T1}{ptm}{m}{n}
		\rput(8.359688,0.1246875){\small $x_3$}
		\usefont{T1}{ptm}{m}{n}
		\rput(5.3196874,-2.6553125){\small $x_4$}
		\usefont{T1}{ptm}{m}{n}
		\rput(3.5596876,-2.6553125){\small $x_5$}
		\usefont{T1}{ptm}{m}{n}
		\rput(0.3996875,0.1646875){\small $x_6$}
		\psline[linewidth=0.024cm](2.1051137,0.40652534)(3.1421876,1.0446875)
		\psline[linewidth=0.024cm](3.7821875,1.0646875)(4.9621873,1.0646875)
		\psline[linewidth=0.024cm](5.6221876,1.0446875)(6.5221877,0.4046875)
		\psline[linewidth=0.024cm](6.5221877,-0.1753125)(5.6021876,-0.9753125)
		\psline[linewidth=0.022cm](4.9621873,-1.0353125)(3.7821875,-1.0353125)
		\psline[linewidth=0.024cm](3.1221876,-0.9553125)(2.1221876,-0.2153125)
		\psline[linewidth=0.024cm](1.1421875,0.3646875)(3.1421876,1.8646874)
		\psline[linewidth=0.024cm](3.7421875,1.8646874)(4.9821873,1.8646874)
		\psline[linewidth=0.024cm](5.5421877,1.8646874)(7.5421877,0.4246875)
		\psline[linewidth=0.024cm](7.5821877,-0.1553125)(5.5821877,-1.9353125)
		\psline[linewidth=0.024cm](4.9821873,-1.9553125)(3.7221875,-1.9753125)
		\psline[linewidth=0.024cm](3.1421876,-1.9353125)(1.1421875,-0.2153125)
		\psbezier[linewidth=0.022](3.6221876,-2.0953126)(4.7821875,-2.1553125)(4.4221873,1.170756)(3.7933187,1.1846875)
		\psbezier[linewidth=0.024](3.6821876,-1.2953125)(4.887579,-1.2953125)(4.8621874,1.9511684)(3.6918037,1.9646875)
		\usefont{T1}{ptm}{m}{n}
		\rput(7.525781,0.0996875){\footnotesize $1$}
		\pscircle[linewidth=0.022,dimen=outer](7.5021877,0.1246875){0.3}
		\pscircle[linewidth=0.022,dimen=outer](6.6921873,0.0946875){0.35}
		\pscircle[linewidth=0.022,dimen=outer](6.6921873,0.0946875){0.25}
		\psline[linewidth=0.004cm](6.3688955,0.19652528)(6.6154795,0.41284972)
		\psline[linewidth=0.004cm](6.353161,0.13019174)(6.731214,0.43918326)
		\psline[linewidth=0.004cm](6.3621874,0.0646875)(6.7821875,0.4246875)
		\psline[linewidth=0.004cm](6.3621874,0.0046875)(6.8221874,0.3846875)
		\psline[linewidth=0.004cm](6.3821874,-0.0353125)(6.8821874,0.3646875)
		\psline[linewidth=0.004cm](6.4021873,-0.0753125)(6.9221873,0.3246875)
		\psline[linewidth=0.004cm](6.4221873,-0.1153125)(6.9621873,0.2846875)
		\psline[linewidth=0.004cm](6.4621873,-0.1353125)(6.9821873,0.2446875)
		\psline[linewidth=0.004cm](6.5021877,-0.1753125)(7.0021877,0.1846875)
		\psline[linewidth=0.004cm](6.5421877,-0.1953125)(7.0221877,0.1446875)
		\psline[linewidth=0.004cm](6.6021876,-0.2153125)(7.0221877,0.0846875)
		\psline[linewidth=0.004cm](6.6821876,-0.2153125)(7.0021877,0.0046875)
		\psline[linewidth=0.004cm](6.7421875,-0.2153125)(6.9621873,-0.0753125)
		\usefont{T1}{ptm}{m}{n}
		\rput(6.6857815,0.0796875){\footnotesize $1$}
		\pscircle[linewidth=0.022,dimen=outer](5.2621875,1.8846875){0.3}
		\usefont{T1}{ptm}{m}{n}
		\rput(5.2857814,1.8996875){\footnotesize $1$}
		\psline[linewidth=0.004cm](5.0021877,1.9846874)(5.2821875,2.1646874)
		\psline[linewidth=0.004cm](4.9621873,1.9046875)(5.3621874,2.1646874)
		\psline[linewidth=0.004cm](4.9821873,1.8646874)(5.3821874,2.1246874)
		\psline[linewidth=0.004cm](4.9821873,1.8246875)(5.4421873,2.1046875)
		\psline[linewidth=0.004cm](5.0021877,1.7846875)(5.4821873,2.0846875)
		\psline[linewidth=0.004cm](5.0221877,1.7446876)(5.5021877,2.0446875)
		\psline[linewidth=0.004cm](5.0421877,1.7046875)(5.5021877,2.0046875)
		\psline[linewidth=0.004cm](5.0821877,1.6846875)(5.5421877,1.9646875)
		\psline[linewidth=0.004cm](5.1221876,1.6646875)(5.5621877,1.9246875)
		\psline[linewidth=0.004cm](5.1621876,1.6446875)(5.5421877,1.8646874)
		\psline[linewidth=0.004cm](5.2021875,1.6246876)(5.5421877,1.8246875)
		\psline[linewidth=0.004cm](5.2621875,1.6246876)(5.5021877,1.7646875)
		\usefont{T1}{ptm}{m}{n}
		\rput(5.3057814,-1.0603125){\footnotesize $1$}
		\pscircle[linewidth=0.022,dimen=outer](5.2721877,-1.0653125){0.35}
		\pscircle[linewidth=0.022,dimen=outer](5.2721877,-1.0653125){0.25}
		\usefont{T1}{ptm}{m}{n}
		\rput(3.4357812,1.8796875){\footnotesize $i$}
		\pscircle[linewidth=0.022,dimen=outer](3.4421875,1.8846875){0.3}
		\end{pspicture} 
	}

	\caption{A desired coloring of $F$ with respect to this cover}
\end{figure}

\begin{figure}[ht]\label{fig4}
	\centering
	\scalebox{1} 
	{
		\begin{pspicture}(0,-2.86875)(8.835,2.86875)
		\pscircle[linewidth=0.022,dimen=outer](3.4721875,-1.0453125){0.35}
		\pscircle[linewidth=0.022,dimen=outer](3.4721875,-1.0453125){0.25}
		\psline[linewidth=0.004cm](3.1488955,-0.9434747)(3.3954797,-0.72715026)
		\psline[linewidth=0.004cm](3.1331608,-1.0098083)(3.511214,-0.70081675)
		\psline[linewidth=0.004cm](3.1421876,-1.0753125)(3.5621874,-0.7153125)
		\psline[linewidth=0.004cm](3.1421876,-1.1353126)(3.6021874,-0.7553125)
		\psline[linewidth=0.004cm](3.1621876,-1.1753125)(3.6621876,-0.7753125)
		\psline[linewidth=0.004cm](3.1821876,-1.2153125)(3.7021875,-0.8153125)
		\psline[linewidth=0.004cm](3.2021875,-1.2553124)(3.7421875,-0.8553125)
		\psline[linewidth=0.004cm](3.2421875,-1.2753125)(3.7621875,-0.8953125)
		\psline[linewidth=0.004cm](3.2821875,-1.3153125)(3.7821875,-0.9553125)
		\psline[linewidth=0.004cm](3.3221874,-1.3353125)(3.8021874,-0.9953125)
		\psline[linewidth=0.004cm](3.3821876,-1.3553125)(3.8021874,-1.0553125)
		\psline[linewidth=0.004cm](3.4621875,-1.3553125)(3.7821875,-1.1353126)
		\psline[linewidth=0.004cm](3.5221875,-1.3553125)(3.7421875,-1.2153125)
		\usefont{T1}{ptm}{m}{n}
		\rput(3.5057812,-1.1003125){\footnotesize $2$}
		\pscircle[linewidth=0.022,dimen=outer](7.5221877,0.1246875){0.3}
		\usefont{T1}{ptm}{m}{n}
		\rput(7.545781,0.1396875){\footnotesize $1$}
		\psline[linewidth=0.004cm](7.2621875,0.2246875)(7.5421877,0.4046875)
		\psline[linewidth=0.004cm](7.2221875,0.1446875)(7.6221876,0.4046875)
		\psline[linewidth=0.004cm](7.2421875,0.1046875)(7.6421876,0.3646875)
		\psline[linewidth=0.004cm](7.2421875,0.0646875)(7.7021875,0.3446875)
		\psline[linewidth=0.004cm](7.2621875,0.0246875)(7.7421875,0.3246875)
		\psline[linewidth=0.004cm](7.2821875,-0.0153125)(7.7621875,0.2846875)
		\psline[linewidth=0.004cm](7.3021874,-0.0553125)(7.7621875,0.2446875)
		\psline[linewidth=0.004cm](7.3421874,-0.0753125)(7.8021874,0.2046875)
		\psline[linewidth=0.004cm](7.3821874,-0.0953125)(7.8221874,0.1646875)
		\psline[linewidth=0.004cm](7.4221873,-0.1153125)(7.8021874,0.1046875)
		\psline[linewidth=0.004cm](7.4621873,-0.1353125)(7.8021874,0.0646875)
		\psline[linewidth=0.004cm](7.5221877,-0.1353125)(7.7621875,0.0046875)
		\usefont{T1}{ptm}{m}{n}
		\rput(5.3057814,-1.1203125){\footnotesize $1$}
		\pscircle[linewidth=0.022,dimen=outer](5.2921877,-1.0853125){0.35}
		\pscircle[linewidth=0.022,dimen=outer](5.2921877,-1.0853125){0.25}
		\psline[linewidth=0.004cm](4.9688954,-0.98347473)(5.2154794,-0.7671503)
		\psline[linewidth=0.004cm](4.953161,-1.0498083)(5.331214,-0.7408167)
		\psline[linewidth=0.004cm](4.9621873,-1.1153125)(5.3821874,-0.7553125)
		\psline[linewidth=0.004cm](4.9621873,-1.1753125)(5.4221873,-0.7953125)
		\psline[linewidth=0.004cm](4.9821873,-1.2153125)(5.4821873,-0.8153125)
		\psline[linewidth=0.004cm](5.0021877,-1.2553124)(5.5221877,-0.8553125)
		\psline[linewidth=0.004cm](5.0221877,-1.2953125)(5.5621877,-0.8953125)
		\psline[linewidth=0.004cm](5.0621877,-1.3153125)(5.5821877,-0.9353125)
		\psline[linewidth=0.004cm](5.1021876,-1.3553125)(5.6021876,-0.9953125)
		\psline[linewidth=0.004cm](5.1421876,-1.3753124)(5.6221876,-1.0353125)
		\psline[linewidth=0.004cm](5.2021875,-1.3953125)(5.6221876,-1.0953125)
		\psline[linewidth=0.004cm](5.2821875,-1.3953125)(5.6021876,-1.1753125)
		\psline[linewidth=0.004cm](5.3421874,-1.3953125)(5.5621877,-1.2553124)
		\pscircle[linewidth=0.022,dimen=outer](6.7121873,0.1146875){0.35}
		\pscircle[linewidth=0.022,dimen=outer](6.7121873,0.1146875){0.25}
		\usefont{T1}{ptm}{m}{n}
		\rput(5.2857814,-1.8803124){\footnotesize $1$}
		\pscircle[linewidth=0.022,dimen=outer](5.3021874,-1.8953125){0.3}
		\rput{89.35016}(7.086884,-6.9383473){\psellipse[linewidth=0.022,linestyle=dashed,dash=0.16cm 0.16cm,dimen=outer](7.0521874,0.1146875)(0.53,0.89)}
		\usefont{T1}{ptm}{m}{n}
		\rput(2.0057812,0.0996875){\footnotesize $1$}
		\pscircle[linewidth=0.022,dimen=outer](1.9721875,0.0946875){0.35}
		\pscircle[linewidth=0.022,dimen=outer](1.9721875,0.0946875){0.25}
		\pscircle[linewidth=0.022,dimen=outer](1.1621875,0.0846875){0.3}
		\usefont{T1}{ptm}{m}{n}
		\rput(1.1857812,0.0996875){\footnotesize $1$}
		\psline[linewidth=0.004cm](0.9021875,0.1846875)(1.1821876,0.3646875)
		\psline[linewidth=0.004cm](0.8621875,0.1046875)(1.2621875,0.3646875)
		\psline[linewidth=0.004cm](0.8821875,0.0646875)(1.2821875,0.3246875)
		\psline[linewidth=0.004cm](0.8821875,0.0246875)(1.3421875,0.3046875)
		\psline[linewidth=0.004cm](0.9021875,-0.0153125)(1.3821875,0.2846875)
		\psline[linewidth=0.004cm](0.9221875,-0.0553125)(1.4021875,0.2446875)
		\psline[linewidth=0.004cm](0.9421875,-0.0953125)(1.4021875,0.2046875)
		\psline[linewidth=0.004cm](0.9821875,-0.1153125)(1.4421875,0.1646875)
		\psline[linewidth=0.004cm](1.0221875,-0.1353125)(1.4621875,0.1246875)
		\psline[linewidth=0.004cm](1.0621876,-0.1553125)(1.4421875,0.0646875)
		\psline[linewidth=0.004cm](1.1021875,-0.1753125)(1.4421875,0.0246875)
		\psline[linewidth=0.004cm](1.1621875,-0.1753125)(1.4021875,-0.0353125)
		\rput{89.35016}(1.6685853,-1.5382434){\psellipse[linewidth=0.022,linestyle=dashed,dash=0.16cm 0.16cm,dimen=outer](1.6121875,0.0746875)(0.53,0.89)}
		\usefont{T1}{ptm}{m}{n}
		\rput(3.5,1.0796875){\footnotesize $j$}
		\pscircle[linewidth=0.022,dimen=outer](3.4721875,1.0946875){0.35}
		\pscircle[linewidth=0.022,dimen=outer](3.4721875,1.0946875){0.25}
		\usefont{T1}{ptm}{m}{n}
		\rput(3.4857812,-1.9003125){\footnotesize $1$}
		\pscircle[linewidth=0.022,dimen=outer](3.4421875,-1.8953125){0.3}
		\pscircle[linewidth=0.022,dimen=outer](3.4421875,1.9446875){0.3}
		\psline[linewidth=0.004cm](3.1821876,2.0446875)(3.4621875,2.2246876)
		\psline[linewidth=0.004cm](3.1421876,1.9646875)(3.5421875,2.2246876)
		\psline[linewidth=0.004cm](3.1621876,1.9246875)(3.5621874,2.1846876)
		\psline[linewidth=0.004cm](3.1621876,1.8846875)(3.6221876,2.1646874)
		\psline[linewidth=0.004cm](3.1821876,1.8446875)(3.6621876,2.1446874)
		\psline[linewidth=0.004cm](3.2021875,1.8046875)(3.6821876,2.1046875)
		\psline[linewidth=0.004cm](3.2221875,1.7646875)(3.6821876,2.0646875)
		\psline[linewidth=0.004cm](3.2621875,1.7446876)(3.7221875,2.0246875)
		\psline[linewidth=0.004cm](3.3021874,1.7246875)(3.7421875,1.9846874)
		\psline[linewidth=0.004cm](3.3421874,1.7046875)(3.7221875,1.9246875)
		\psline[linewidth=0.004cm](3.3821876,1.6846875)(3.7221875,1.8846875)
		\psline[linewidth=0.004cm](3.4421875,1.6846875)(3.6821876,1.8246875)
		\usefont{T1}{ptm}{m}{n}
		\rput(3.4357812,1.9396875){\footnotesize $i$}
		\psellipse[linewidth=0.022,linestyle=dashed,dash=0.16cm 0.16cm,dimen=outer](3.4521875,-1.4653125)(0.53,0.89)
		\psellipse[linewidth=0.022,linestyle=dashed,dash=0.16cm 0.16cm,dimen=outer](5.2921877,-1.4653125)(0.53,0.89)
		\psellipse[linewidth=0.022,linestyle=dashed,dash=0.16cm 0.16cm,dimen=outer](3.4521875,1.5146875)(0.53,0.89)
		\psellipse[linewidth=0.022,linestyle=dashed,dash=0.16cm 0.16cm,dimen=outer](5.2921877,1.5346875)(0.53,0.89)
		\usefont{T1}{ptm}{m}{n}
		\rput(3.4596875,2.6846876){\small $x_1$}
		\usefont{T1}{ptm}{m}{n}
		\rput(5.2596874,2.6846876){\small $x_2$}
		\usefont{T1}{ptm}{m}{n}
		\rput(8.359688,0.1246875){\small $x_3$}
		\usefont{T1}{ptm}{m}{n}
		\rput(5.3196874,-2.6553125){\small $x_4$}
		\usefont{T1}{ptm}{m}{n}
		\rput(3.5596876,-2.6553125){\small $x_5$}
		\usefont{T1}{ptm}{m}{n}
		\rput(0.3996875,0.1646875){\small $x_6$}
		\psline[linewidth=0.024cm](2.1051137,0.40652534)(3.1421876,1.8846875)
		\psline[linewidth=0.024cm](3.7821875,1.0646875)(4.9621873,1.0646875)
		\psline[linewidth=0.024cm](5.6221876,1.0446875)(6.5221877,0.4046875)
		\psline[linewidth=0.024cm](6.5221877,-0.1753125)(5.6021876,-0.9753125)
		\psline[linewidth=0.08cm](4.9621873,-1.0353125)(3.7821875,-1.0353125)
		\psline[linewidth=0.024cm](3.1221876,-0.9553125)(2.1221876,-0.2153125)
		\psline[linewidth=0.024cm](1.1621875,0.3846875)(3.1621876,1.0846875)
		\psline[linewidth=0.024cm](3.7421875,1.8646874)(4.9821873,1.8646874)
		\psline[linewidth=0.024cm](5.5421877,1.8646874)(7.5421877,0.4246875)
		\psline[linewidth=0.024cm](7.5821877,-0.1553125)(5.5821877,-1.9353125)
		\psline[linewidth=0.024cm](4.9821873,-1.9553125)(3.7221875,-1.9753125)
		\psline[linewidth=0.024cm](3.1421876,-1.9353125)(1.1421875,-0.2153125)
		\psbezier[linewidth=0.022](3.7771952,-1.2053437)(4.157441,-1.2153381)(4.1621876,1.1847245)(3.7819414,1.1947188)
		\psbezier[linewidth=0.024](3.677195,-2.058559)(4.847579,-2.0753126)(4.8621874,1.9479338)(3.6918037,1.9646875)
		\usefont{T1}{ptm}{m}{n}
		\rput(6.7257814,0.0796875){\footnotesize $1$}
		\usefont{T1}{ptm}{m}{n}
		\rput(5.3257813,1.9396875){\footnotesize $1$}
		\pscircle[linewidth=0.022,dimen=outer](5.3021874,1.9646875){0.3}
		\pscircle[linewidth=0.022,dimen=outer](5.2921877,1.1346875){0.35}
		\pscircle[linewidth=0.022,dimen=outer](5.2921877,1.1346875){0.25}
		\psline[linewidth=0.004cm](4.9688954,1.2365253)(5.2154794,1.4528497)
		\psline[linewidth=0.004cm](4.953161,1.1701918)(5.331214,1.4791833)
		\psline[linewidth=0.004cm](4.9621873,1.1046875)(5.3821874,1.4646875)
		\psline[linewidth=0.004cm](4.9621873,1.0446875)(5.4221873,1.4246875)
		\psline[linewidth=0.004cm](4.9821873,1.0046875)(5.4821873,1.4046875)
		\psline[linewidth=0.004cm](5.0021877,0.9646875)(5.5221877,1.3646874)
		\psline[linewidth=0.004cm](5.0221877,0.9246875)(5.5621877,1.3246875)
		\psline[linewidth=0.004cm](5.0621877,0.9046875)(5.5821877,1.2846875)
		\psline[linewidth=0.004cm](5.1021876,0.8646875)(5.6021876,1.2246875)
		\psline[linewidth=0.004cm](5.1421876,0.8446875)(5.6221876,1.1846875)
		\psline[linewidth=0.004cm](5.2021875,0.8246875)(5.6221876,1.1246876)
		\psline[linewidth=0.004cm](5.2821875,0.8246875)(5.6021876,1.0446875)
		\psline[linewidth=0.004cm](5.3421874,0.8246875)(5.5621877,0.9646875)
		\usefont{T1}{ptm}{m}{n}
		\rput(5.2857814,1.1196876){\footnotesize $1$}
		\end{pspicture} 
	}
	
	\caption{A desired coloring of $F$ with respect to this cover}
\end{figure}

\begin{figure}[ht]\label{fig5}
	\centering
	\scalebox{1} 
	{
		\begin{pspicture}(0,-2.86875)(8.835,2.86875)
		\pscircle[linewidth=0.022,dimen=outer](3.4721875,-1.0453125){0.35}
		\pscircle[linewidth=0.022,dimen=outer](3.4721875,-1.0453125){0.25}
		\psline[linewidth=0.004cm](3.1488955,-0.9434747)(3.3954797,-0.72715026)
		\psline[linewidth=0.004cm](3.1331608,-1.0098083)(3.511214,-0.70081675)
		\psline[linewidth=0.004cm](3.1421876,-1.0753125)(3.5621874,-0.7153125)
		\psline[linewidth=0.004cm](3.1421876,-1.1353126)(3.6021874,-0.7553125)
		\psline[linewidth=0.004cm](3.1621876,-1.1753125)(3.6621876,-0.7753125)
		\psline[linewidth=0.004cm](3.1821876,-1.2153125)(3.7021875,-0.8153125)
		\psline[linewidth=0.004cm](3.2021875,-1.2553124)(3.7421875,-0.8553125)
		\psline[linewidth=0.004cm](3.2421875,-1.2753125)(3.7621875,-0.8953125)
		\psline[linewidth=0.004cm](3.2821875,-1.3153125)(3.7821875,-0.9553125)
		\psline[linewidth=0.004cm](3.3221874,-1.3353125)(3.8021874,-0.9953125)
		\psline[linewidth=0.004cm](3.3821876,-1.3553125)(3.8021874,-1.0553125)
		\psline[linewidth=0.004cm](3.4621875,-1.3553125)(3.7821875,-1.1353126)
		\psline[linewidth=0.004cm](3.5221875,-1.3553125)(3.7421875,-1.2153125)
		\usefont{T1}{ptm}{m}{n}
		\rput(3.5057812,-1.1003125){\footnotesize $2$}
		\usefont{T1}{ptm}{m}{n}
		\rput(3.5,1.1196876){\footnotesize $j$}
		\pscircle[linewidth=0.022,dimen=outer](3.4521875,1.0946875){0.35}
		\pscircle[linewidth=0.022,dimen=outer](3.4521875,1.0946875){0.25}
		\psline[linewidth=0.004cm](3.1288955,1.1965252)(3.3754795,1.4128498)
		\psline[linewidth=0.004cm](3.1131608,1.1301917)(3.491214,1.4391832)
		\psline[linewidth=0.004cm](3.1221876,1.0646875)(3.5421875,1.4246875)
		\psline[linewidth=0.004cm](3.1221876,1.0046875)(3.5821874,1.3846875)
		\psline[linewidth=0.004cm](3.1421876,0.9646875)(3.6421876,1.3646874)
		\psline[linewidth=0.004cm](3.1621876,0.9246875)(3.6821876,1.3246875)
		\psline[linewidth=0.004cm](3.1821876,0.8846875)(3.7221875,1.2846875)
		\psline[linewidth=0.004cm](3.2221875,0.8646875)(3.7421875,1.2446876)
		\psline[linewidth=0.004cm](3.2621875,0.8246875)(3.7621875,1.1846875)
		\psline[linewidth=0.004cm](3.3021874,0.8046875)(3.7821875,1.1446875)
		\psline[linewidth=0.004cm](3.3621874,0.7846875)(3.7821875,1.0846875)
		\psline[linewidth=0.004cm](3.4421875,0.7846875)(3.7621875,1.0046875)
		\psline[linewidth=0.004cm](3.5021875,0.7846875)(3.7221875,0.9246875)
		\usefont{T1}{ptm}{m}{n}
		\rput(5.2657814,-1.9603125){\footnotesize $1$}
		\rput{89.35016}(7.086884,-6.9383473){\psellipse[linewidth=0.022,linestyle=dashed,dash=0.16cm 0.16cm,dimen=outer](7.0521874,0.1146875)(0.53,0.89)}
		\rput{89.35016}(1.6685853,-1.5382434){\psellipse[linewidth=0.022,linestyle=dashed,dash=0.16cm 0.16cm,dimen=outer](1.6121875,0.0746875)(0.53,0.89)}
		\usefont{T1}{ptm}{m}{n}
		\rput(5.2857814,1.1196876){\footnotesize $1$}
		\pscircle[linewidth=0.022,dimen=outer](5.2921877,1.1346875){0.35}
		\pscircle[linewidth=0.022,dimen=outer](5.2921877,1.1346875){0.25}
		\usefont{T1}{ptm}{m}{n}
		\rput(3.4857812,-1.9003125){\footnotesize $1$}
		\pscircle[linewidth=0.022,dimen=outer](3.4421875,-1.8953125){0.3}
		\pscircle[linewidth=0.022,dimen=outer](5.2821875,-1.9553125){0.3}
		\psline[linewidth=0.004cm](5.0221877,-1.8553125)(5.3021874,-1.6753125)
		\psline[linewidth=0.004cm](4.9821873,-1.9353125)(5.3821874,-1.6753125)
		\psline[linewidth=0.004cm](5.0021877,-1.9753125)(5.4021873,-1.7153125)
		\psline[linewidth=0.004cm](5.0021877,-2.0153124)(5.4621873,-1.7353125)
		\psline[linewidth=0.004cm](5.0221877,-2.0553124)(5.5021877,-1.7553124)
		\psline[linewidth=0.004cm](5.0421877,-2.0953126)(5.5221877,-1.7953125)
		\psline[linewidth=0.004cm](5.0621877,-2.1353126)(5.5221877,-1.8353125)
		\psline[linewidth=0.004cm](5.1021876,-2.1553125)(5.5621877,-1.8753124)
		\psline[linewidth=0.004cm](5.1421876,-2.1753125)(5.5821877,-1.9153125)
		\psline[linewidth=0.004cm](5.1821876,-2.1953125)(5.5621877,-1.9753125)
		\psline[linewidth=0.004cm](5.2221875,-2.2153125)(5.5621877,-2.0153124)
		\psline[linewidth=0.004cm](5.2821875,-2.2153125)(5.5221877,-2.0753126)
		\psellipse[linewidth=0.022,linestyle=dashed,dash=0.16cm 0.16cm,dimen=outer](3.4521875,-1.4653125)(0.53,0.89)
		\psellipse[linewidth=0.022,linestyle=dashed,dash=0.16cm 0.16cm,dimen=outer](5.2921877,-1.4653125)(0.53,0.89)
		\psellipse[linewidth=0.022,linestyle=dashed,dash=0.16cm 0.16cm,dimen=outer](3.4521875,1.5146875)(0.53,0.89)
		\psellipse[linewidth=0.022,linestyle=dashed,dash=0.16cm 0.16cm,dimen=outer](5.2921877,1.5346875)(0.53,0.89)
		\usefont{T1}{ptm}{m}{n}
		\rput(3.4596875,2.6846876){\small $x_1$}
		\usefont{T1}{ptm}{m}{n}
		\rput(5.2596874,2.6846876){\small $x_2$}
		\usefont{T1}{ptm}{m}{n}
		\rput(8.359688,0.1246875){\small $x_3$}
		\usefont{T1}{ptm}{m}{n}
		\rput(5.3196874,-2.6553125){\small $x_4$}
		\usefont{T1}{ptm}{m}{n}
		\rput(3.5596876,-2.6553125){\small $x_5$}
		\usefont{T1}{ptm}{m}{n}
		\rput(0.3996875,0.1646875){\small $x_6$}
		\psline[linewidth=0.024cm](2.1051137,0.40652534)(3.1821876,1.7446876)
		\psline[linewidth=0.024cm](3.7821875,1.0646875)(4.9621873,1.0646875)
		\psline[linewidth=0.024cm](5.6221876,1.0446875)(6.5221877,0.4046875)
		\psline[linewidth=0.024cm](6.5221877,-0.1753125)(5.6021876,-0.9753125)
		\psline[linewidth=0.022cm](4.9621873,-1.0353125)(3.7821875,-1.0353125)
		\psline[linewidth=0.08cm](3.1221876,-0.9553125)(2.1221876,-0.2153125)
		\psline[linewidth=0.024cm](1.1421875,0.3646875)(3.1021874,1.1046875)
		\psline[linewidth=0.024cm](3.7421875,1.8646874)(4.9821873,1.8646874)
		\psline[linewidth=0.024cm](5.5421877,1.8646874)(7.5421877,0.4246875)
		\psline[linewidth=0.024cm](7.5821877,-0.1553125)(5.5821877,-1.9353125)
		\psline[linewidth=0.024cm](4.9821873,-1.9553125)(3.7221875,-1.9753125)
		\psline[linewidth=0.024cm](3.1421876,-1.9353125)(1.1421875,-0.2153125)
		\psbezier[linewidth=0.022](3.6221876,-2.0953126)(4.7821875,-2.1553125)(4.4221873,1.170756)(3.7933187,1.1846875)
		\psbezier[linewidth=0.024](3.6821876,-1.2953125)(4.887579,-1.2953125)(4.8621874,1.9511684)(3.6918037,1.9646875)
		\usefont{T1}{ptm}{m}{n}
		\rput(7.525781,0.0996875){\footnotesize $1$}
		\pscircle[linewidth=0.022,dimen=outer](7.5021877,0.1246875){0.3}
		\pscircle[linewidth=0.022,dimen=outer](6.6921873,0.0946875){0.35}
		\pscircle[linewidth=0.022,dimen=outer](6.6921873,0.0946875){0.25}
		\psline[linewidth=0.004cm](6.3688955,0.19652528)(6.6154795,0.41284972)
		\psline[linewidth=0.004cm](6.353161,0.13019174)(6.731214,0.43918326)
		\psline[linewidth=0.004cm](6.3621874,0.0646875)(6.7821875,0.4246875)
		\psline[linewidth=0.004cm](6.3621874,0.0046875)(6.8221874,0.3846875)
		\psline[linewidth=0.004cm](6.3821874,-0.0353125)(6.8821874,0.3646875)
		\psline[linewidth=0.004cm](6.4021873,-0.0753125)(6.9221873,0.3246875)
		\psline[linewidth=0.004cm](6.4221873,-0.1153125)(6.9621873,0.2846875)
		\psline[linewidth=0.004cm](6.4621873,-0.1353125)(6.9821873,0.2446875)
		\psline[linewidth=0.004cm](6.5021877,-0.1753125)(7.0021877,0.1846875)
		\psline[linewidth=0.004cm](6.5421877,-0.1953125)(7.0221877,0.1446875)
		\psline[linewidth=0.004cm](6.6021876,-0.2153125)(7.0221877,0.0846875)
		\psline[linewidth=0.004cm](6.6821876,-0.2153125)(7.0021877,0.0046875)
		\psline[linewidth=0.004cm](6.7421875,-0.2153125)(6.9621873,-0.0753125)
		\usefont{T1}{ptm}{m}{n}
		\rput(6.6857815,0.0796875){\footnotesize $1$}
		\pscircle[linewidth=0.022,dimen=outer](5.2621875,1.8846875){0.3}
		\usefont{T1}{ptm}{m}{n}
		\rput(5.2857814,1.8996875){\footnotesize $1$}
		\psline[linewidth=0.004cm](5.0021877,1.9846874)(5.2821875,2.1646874)
		\psline[linewidth=0.004cm](4.9621873,1.9046875)(5.3621874,2.1646874)
		\psline[linewidth=0.004cm](4.9821873,1.8646874)(5.3821874,2.1246874)
		\psline[linewidth=0.004cm](4.9821873,1.8246875)(5.4421873,2.1046875)
		\psline[linewidth=0.004cm](5.0021877,1.7846875)(5.4821873,2.0846875)
		\psline[linewidth=0.004cm](5.0221877,1.7446876)(5.5021877,2.0446875)
		\psline[linewidth=0.004cm](5.0421877,1.7046875)(5.5021877,2.0046875)
		\psline[linewidth=0.004cm](5.0821877,1.6846875)(5.5421877,1.9646875)
		\psline[linewidth=0.004cm](5.1221876,1.6646875)(5.5621877,1.9246875)
		\psline[linewidth=0.004cm](5.1621876,1.6446875)(5.5421877,1.8646874)
		\psline[linewidth=0.004cm](5.2021875,1.6246876)(5.5421877,1.8246875)
		\psline[linewidth=0.004cm](5.2621875,1.6246876)(5.5021877,1.7646875)
		\usefont{T1}{ptm}{m}{n}
		\rput(5.3057814,-1.0603125){\footnotesize $1$}
		\pscircle[linewidth=0.022,dimen=outer](5.2721877,-1.0653125){0.35}
		\pscircle[linewidth=0.022,dimen=outer](5.2721877,-1.0653125){0.25}
		\usefont{T1}{ptm}{m}{n}
		\rput(3.4357812,1.8796875){\footnotesize $i$}
		\pscircle[linewidth=0.022,dimen=outer](3.4421875,1.8846875){0.3}
		\pscircle[linewidth=0.022,dimen=outer](2.0121875,0.0946875){0.35}
		\pscircle[linewidth=0.022,dimen=outer](2.0121875,0.0946875){0.25}
		\psline[linewidth=0.004cm](1.6888955,0.19652528)(1.9354795,0.41284972)
		\psline[linewidth=0.004cm](1.6731609,0.13019174)(2.0512142,0.43918326)
		\psline[linewidth=0.004cm](1.6821876,0.0646875)(2.1021874,0.4246875)
		\psline[linewidth=0.004cm](1.6821876,0.0046875)(2.1421876,0.3846875)
		\psline[linewidth=0.004cm](1.7021875,-0.0353125)(2.2021875,0.3646875)
		\psline[linewidth=0.004cm](1.7221875,-0.0753125)(2.2421875,0.3246875)
		\psline[linewidth=0.004cm](1.7421875,-0.1153125)(2.2821875,0.2846875)
		\psline[linewidth=0.004cm](1.7821875,-0.1353125)(2.3021874,0.2446875)
		\psline[linewidth=0.004cm](1.8221875,-0.1753125)(2.3221874,0.1846875)
		\psline[linewidth=0.004cm](1.8621875,-0.1953125)(2.3421874,0.1446875)
		\psline[linewidth=0.004cm](1.9221874,-0.2153125)(2.3421874,0.0846875)
		\psline[linewidth=0.004cm](2.0021875,-0.2153125)(2.3221874,0.0046875)
		\psline[linewidth=0.004cm](2.0621874,-0.2153125)(2.2821875,-0.0753125)
		\usefont{T1}{ptm}{m}{n}
		\rput(2.0057812,0.0796875){\footnotesize $1$}
		\usefont{T1}{ptm}{m}{n}
		\rput(1.2057812,0.0396875){\footnotesize $1$}
		\pscircle[linewidth=0.022,dimen=outer](1.1821876,0.0646875){0.3}
		\end{pspicture} 
	}
	
	\caption{A desired coloring of $F$ with respect to this cover}
\end{figure}
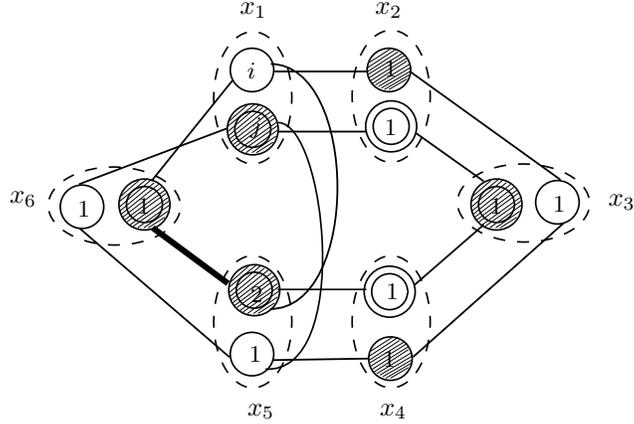

Now we are ready to prove a general result.

\begin{theorem}\label{main2} Every planar graph without $4$-cycles adjacent to $3$-cycles is 
DP-$(f_1, \ldots, f_s)$-colorable 
if $s \geq 2,$ $f_1(v)+\cdots+f_s(v)\geq 4$ for each vertex $v,$ 
and $f_i(v)\in \{0,1,2\}$ for each $v$ and $i.$ 
\end{theorem} 

\begin{proof}
Suppose that $G$ with a cover $H$ is a minimal counterexample. 
Since $G$ does not have $4$-cycles adjacent to $3$-cycles, 
$G$ has a configuration $F$ as in Lemma \ref{C35}. 
By minimality, there is a DP-$(f_1, \ldots, f_s)$-coloring $R'$ on $G - \{x_1,\ldots,x_6\}.$ 

For each $x_k \in V(F)$ and $k \in \{1,\ldots,s\},$ we put $f^*_i(x_k)$ equals $f_i(x_k)$ minus 
the number of $(v,j) \in R'$ such that $(v,j)$ and $(x,i)$ are adjacent in $H.$ 

similar to the proof of Theorem \ref{main1}, 
if we have  a DP-$(f^*_1,\ldots,f^*_s)$-coloring of $F,$ 
then one can obtain a desired DP-$(f_1,\ldots,f_s)$-coloring on $G.$ 

Note that each $x_i$ may have different size of its list of colors. 
To make all $x_k$s have comparable $(f_1^*(x_k),\ldots,f_s^*(x_k)),$ 
we fill out illegal color $i$ for $x_k$ by using $f_i^*(x_k)=0.$ 
By the definition and conditions of $f_i^*,$ 
initially $(f_1^*(x_k),\ldots,f_s^*(x_k))$ has one or two positive coordinates when $k\in \{2,3,4,6\}$ 
and $(f_1^*(x_k),\ldots,f_s^*(x_k))$ has two or three positive coordinates when $k\in \{1,5\}.$ 
If $(f_1^*(x_k),\ldots,f_s^*(x_k))$ and $(f_1^*(x_{k+1}),\ldots,f_s^*(x_{k+1}))$ 
have different numbers of positive coordinates, 
then we can complete the coloring by a method similar to CASE 2 in a proof of Theorem \ref{main1}. 

Thus we assume that each $(f_1^*(x_k),\ldots,f_s^*(x_k))$ has exactly two positive coordinates. 
Since  color $i$ 
in which $f_i^*(x_k)=0$ can be discarded from consideration, 
we arrive that  each $(f_1^*(x_k),\ldots,f_s^*(x_k))$ can be reduced to $(f_{i_1}^*(x_k),f_{i_2}^*(x_k)).$ 
Thus the proof can be completed by a method similar to CASE 1 in the proof of Theorem \ref{main1}. 
\end{proof}

\end{document}